\documentclass[final]{amsart}
\usepackage{graphicx}
\usepackage{amsmath}
\usepackage{amsfonts}
\usepackage{amssymb}
\usepackage{color}%?? can delete this before submission
\usepackage[normalem]{ulem}%??
\usepackage{hyperref}%?? can delete this before submission
\usepackage{mathrsfs}

 \newtheorem{thm}{Theorem}%[section]
 \newtheorem{cor}[thm]{Corollary}
 \newtheorem{lem}[thm]{Lemma}
 \newtheorem{prop}[thm]{Proposition}
 \theoremstyle{definition}
 \newtheorem{defn}[thm]{Definition}
 \newtheorem{ex}[thm]{Example}
 \newtheorem{rem}[thm]{Remark}
\newcommand{\eq} [1] {\begin{equation}\label{#1}\quad}
\newcommand{\en} {\end{equation}}
\newcommand{\scal}[1]{\langle#1\rangle}

\newcommand{\tr}{\operatorname{trace}}

\newcommand{\C}{\mathbb{C}}
\newcommand{\F}{\mathbb{F}}
\newcommand{\Fext}{\F_{q^2}}
\newcommand{\R}{\mathbb{R}}
\newcommand{\Z}{\mathbb{Z}}

\newcommand{\associated}{base }

\newcommand{\mf}[1]{\mathbf{#1}}
\newcommand{\vv}{\mathbf{v}}
\newcommand{\uu}{\mathbf{u}}

\newcommand{\ww}{\mathbf{w}}
\newcommand{\M}{{M_n(\F_{q^2})}}
\newcommand{\MM}{{M_2(\F_{q^2})}}
\newcommand{\gfw}{{\Gamma_{F}^\wedge}}
\newcommand{\gf}{{\Gamma_{F}}}
\newcommand{\Real}{{\textmd{Re }}}
\newcommand{\ve}{{\textbf{vec}}}

\usepackage[dvipsnames]{xcolor}

\begin{document}
%--------------------------------------------------------------------------
% editorial commands
%
%\firstpage{1} \issuenumber{4} \volumeandyear{41 (2001)}
%\commby{inhouse}
%\submitted{June 3, 2021}
%\received{??}
%\revised{July 7, 2021}
%\accepted{??}
%---------------------------------------------------------------------------
%Insert here the title, affiliations and abstract:
%

\title{On the geometry of numerical ranges over finite fields}

\author[Camenga]{Kristin A.~Camenga}
\address{Department of Mathematics\\
Juniata College\\
Huntingdon, PA\\ USA}
\email{camenga@juniata.edu}

\author[Collins]{Brandon Collins}
\address{Department of Mathematics\\ University of Nebraska at Omaha\\
6001 Dodge Street\\
Omaha, NE 68182\\
USA}
\email{brandoncollins@unomaha.edu}

\author[Hoefer]{Gage Hoefer}
\address{Department of Mathematics\\ University of Nebraska at Omaha\\
6001 Dodge Street\\
Omaha, NE 68182\\
USA}
\email{ghoefer@unomaha.edu}

\author[Quezada]{Jonny Quezada}
\address{Department of Mathematics\\ University of Nebraska at Omaha\\
6001 Dodge Street\\
Omaha, NE 68182\\
USA}
\email{jonnyquezada@unomaha.edu}

\author[Rault]{Patrick X.~Rault}
\address{Department of Mathematics\\ University of Nebraska at Omaha\\
6001 Dodge Street\\
Omaha, NE 68182\\
USA}
\email{prault@unomaha.edu}

\author[Willson]{James Willson}
\address{Department of Mathematics\\ University of Nebraska at Omaha\\
6001 Dodge Street\\
Omaha, NE 68182\\
USA}
\email{jwillson@unomaha.edu}

\author[Yates]{Rebekah B.~Johnson Yates}
\address{Department of Mathematics\\Houghton College\\1 Willard Ave.\\Houghton, NY 14744\\USA}
\email{rebekah.yates@houghton.edu}

%\author[]{other authors to be added}

\thanks{Work by the second [BC] and sixth [JW] authors has been supported by the University of Nebraska at Omaha Kerrigan Research Minigrants Program.  Work by the first [KAC], fifth [PXR], and seventh [RBJY], authors has been supported by the American Institute of Mathematics REUF continuation program (NSF-DMS 1620073).  The authors thank Kazuki Makino of Juniata College for his honors thesis work which led to insights on the conditions needed for Lemmas \ref{lem:eigenvalue} and \ref{lem:vert}.
}
%----------classification, keywords, date
\subjclass{Primary 15A60}

\keywords{Numerical range, field of values, $2\times 2$ matrices, boundary generating curve, finite fields}

%----------additions
\dedicatory{}

\begin{abstract}
 Numerical ranges over a certain family of finite fields were classified in 2016 by a team including our fifth author \cite{CJKLR}.  Soon afterward Ballico generalized these results to all finite fields and published some new results about the cardinality of the finite field numerical range \cite{Ballico1, Ballico2}.  In this paper we study the geometry of these finite fields using the boundary generating curve, first introduced by Kippenhahn in 1951 \cite{Ki, Ki08}. We restrict our study to square matrices of dimension 2, with at least one eigenvalue in $\mathbb F_{q^2}$.
 \end{abstract}

\maketitle
%Last edited: \today 

\section{Introduction}\label{s:pre}

Let $q$ be an odd prime power and $\F_{q}$ the finite field with $q$ elements.  Note that $\F_{q^2}$ is a 2-dimensional vector space over $\F_q$, similar to $\C$ over $\R$.  Furthermore, there exists some $\alpha$ in $\F_q$ which is not a perfect square, and a $\beta$ in $\F_{q^2}$ such that $\beta^2=\alpha$, and we have $\F_{q^2}=\F_q[\beta]$ as an algebraic field extension. The Frobenius morphism $x\mapsto x^q$ on $\F_{q^2}$ preserves $\F_q$ and behaves like complex conjugation in $\C$; for simplicity we write in this paper $\overline x := x^q$ for any $x$ in $\F_{q^2}$.  We use this to define the conjugate transpose of a matrix $A$ in $\M$:  $A^*$ is the matrix obtained by transposing $A$ and applying the Frobenius morphism to each entry of $A$.  We define the Hermitian form $\scal{\cdot,\cdot}$ on $\F_{q^2}^n$ by $\scal{\mathbf u,\mathbf v}:=\mathbf v^* \mathbf u$.

\begin{defn}
For a matrix $A\in \M$, the \textbf{numerical range of $A$}, as defined in \cite{CJKLR}, is the set
\begin{equation*}
W(A)=\left\{\scal{A\textbf v,\textbf v}\colon \scal{\textbf v,\textbf v}=1, \textbf v\in \F_{q^2}^n\right\}.
\end{equation*}
\end{defn}

The classical complex number numerical range was first studied by Kippenhahn \cite{Ki, Ki08} and has been completely classified for $n\times n$ matrices up to dimension $n=4$ \cite{KRS, ChiNa12, CDRSSY}.  The most recent of these classifications relies on the boundary generating curve of the matrix \cite{ChiNa12, CDRSSY}.  We will now define this curve in detail, as it is pivotal to the results in this paper.  {Note that the first part of this next definition builds on how we traditionally define the real (or $\F_q$) and imaginary (or $\beta \F_q$) parts of a number $x$, as $\textmd{Re }x=(x+\overline x)/2$ and $\textmd{Im }x=(x-\overline x)/(2\beta)$.}

\begin{defn}\label{defn:assoccurve} Let $A\in \M$ and define $\textmd{Re }A=\frac{A+A^*}{2}$, and $\textmd{Im } A=\frac{A-A^*}{2\beta}$; here $\beta\in \F_{q^2}$ such that $\F_{q^2}=\F_q[\beta]$.  For simplicity we write $H_1=\textmd{Re }A$ and $H_2=\textmd{Im }A$, and note that $H_1$ and $H_2$ are both \textbf{Hermitian}, in the sense that for a Hermitian matrix $H$ we have $H^*=H$.  Note further that $A=H_1+\beta H_2$, and that $H_1$ is called the \textbf{Hermitian part} of $A$.

Let
\begin{equation*}
F_A(x:y:t) = \det\left(xH_1 + yH_2+tI_n\right),
\end{equation*}
which is a degree $n$ homogeneous polynomial, and let $\Gamma_{F_A}$ denote its zero set in projective space $\mathbb{P}^2(\F_{q^2})$.  $F_A$ is called the \textbf{\associated polynomial} and $\Gamma_{F_A}$ is called the \textbf{\associated curve} for the matrix $A$. The dual curve to $\Gamma_{F_A}$, denoted $\Gamma^\wedge_{F_A}$, is called the \textbf{boundary generating curve}. When $A$ is clear, we will write $F = F_A$, $\Gamma_{F_A} = \Gamma _F$ and $\Gamma^\wedge_{F_A} = \Gamma^\wedge_{F}$. Additionally, for a line $L$ tangent to $\Gamma_{F_A}$ (respectively, point $P$ on $\Gamma_{F_A}$) we write $\widehat L$ for its dual point on $\gfw$ (respectively, tangent line $\widehat P$ to $\gfw$); note that this duality notation works in the reverse direction as well.\end{defn}

 While $F_A$ will have coefficients in the base field $\F_q$ for any $n\times n$ matrix $A$ since it is the determinant of a Hermitian matrix,  $H_1$ and $H_2$ do not necessarily have coefficients in the base field $\mathbb F_q$.  Indeed, in the two-dimensional case such a Hermitian matrix will take the form
$$\begin{bmatrix}
a & b+\beta c \\
b-\beta c & d
\end{bmatrix},$$
for some $a,b,c,d\in \F_q$.

The $\F_q$-affine points on the boundary generating curve have a natural embedding into $\F_{q}^2$, which is isomorphic to $\F_{q^2}$ as a vector space.  Kippenhahn showed that in the situation of classical numerical ranges over $\C$, the numerical range is the convex hull of this curve \cite{Ki, Ki08}.

For much of the remainder of the paper we will utilize tools that require us to restrict to the situation of square matrices of dimension 2 with at least one eigenvalue in $\mathbb F_{q^2}$. We will prove the following result, which mirrors a result of Kippenhahn's in the classical situation of complex number numerical ranges. Note that our proof technique also applies in the classical complex number situation.

\begin{prop}\label{thm:boundary}
  Let $q$ be an odd prime power, let $A\in M_2(\F_{q^2})$, and assume that $F_A$ is nonsingular.  Then the boundary generating curve is a subset of the numerical range:  $\gfw\subseteq W(A)$.
\end{prop}

{The assumption that $F_A$ is nonsingular mostly rules out the cases when $A$ is unitarily reducible.  However, we will see in Section \ref{s:exceptional} some other situations where this polynomial and the curve $\Gamma_F$ that it defines can be singular.}

While this curve does not define a boundary in the same sense as over the complex numbers, it does have some geometric significance.  In particular, we will show the following in the case of dimension two.  Note that in the classical complex number situation, the numerical range will be an elliptical disc, with boundary generating curve the boundary ellipse.  In this situation our proof technique can be applied to show that the numerical range is indeed a set of scalings of the boundary generating curve; we leave this to the reader.

Lastly, note also that while \textbf{unitary} is a common description of a matrix, we use the term in the following theorem to refer to a unit scalar $u$ with the special property that $u\overline u = |u|^2 = 1$; later we will use $\mathcal U$ to refer to the set of these unitary scalars. Note that since $|u|^2 = u^{q + 1} = 1$ has $q + 1$ solutions (\cite[Theorem 1.5 (iii), page 4]{Hirschfeld} and \cite[Theorem 2.47(ii) and Theorem 2.49]{Lidl}), we have $|\mathcal{U}| = q + 1$.

The result below establishes the final element of geometry of these numerical ranges.

\begin{thm}\label{thm:scalings}
    Let $q$ be an odd prime power, let $A\in M_2(\F_{q^2})$, and assume that {$F_A$ is nonsingular.}  {Furthermore, assume that there is at least one eigenvalue in $\F_{q^2}$ with a corresponding eigenvector $\uu$ of A such that $\scal{\uu,\uu}\neq 0$.}
  Then $\gfw$ is a Hermitian ellipse or hyperbola of the form $\mathcal C_1$, with (after transformations) $\mathcal C_i$ defined as
  $$\mathcal C_i:\frac{(x-c)^2}{a}-\alpha \frac{y^2}{b}=d_i$$ for some $a,b,c,d_i\in \F_q$ with $ab\neq 0$.
   Furthermore, there exist scalars $d_i\in \F_q$ such that $$W(A)=\bigcup\limits_{i=1}^{(q+1)/2} \mathcal C_{i},$$ where elements of the union are disjoint.
   When $\mathcal C_{1}$ is an ellipse, we have $d_2=0$, $|\mathcal C_{i}|=q+1$ for $i\neq 2$, and there are exactly $(q^2+1)/2$ points in $W(A)$.
   When $\mathcal C_{1}$ is a hyperbola, we have $|C_i|=q-1$ for each $i$, and there are exactly $(q^2-1)/2$ points in $W(A)$.
\end{thm}

Note that the second assumption in this proposition was necessary to account for some complications over finite fields, as the Hermitian form does not define an inner product so $\mathbf x\mapsto \scal{\mathbf x,\mathbf x}$ does not define a norm on $\mathbb F_{q^2}$.  Note also that we will drop the adjective Hermitian from our conics for the remainder of this paper.

While it is not the case that $W(A)$ is the convex hull of $\gfw$, this theorem does establish a similar geometric relationship.  Lastly, we combine the two previous results into the following statement about density.  The reader may note that our proof technique can also be used to prove that in the classical situation of complex number numerical ranges, the density is one on the boundary.

\begin{thm}\label{thm:densityF}
  Let $q$ be an odd prime power, let $A\in M_2(\F_{q^2})$, and assume that {$F_A$ is nonsingular.}  {Furthermore, assume that there is at least one eigenvalue in $\F_{q^2}$ with a corresponding eigenvector $\uu$ of A such that $\scal{\uu,\uu}\neq 0$.}
Let $z\in W(A)$ and let $S_z=\{\mathbf v\mid \scal{A \mathbf v, \mathbf v}=z, \; \scal{\vv,\vv}=1\}$, i.e. the pre-image of the numerical range map at $z$. Let $\mathcal U=\{k+\beta \ell\in \F_{q^2}\mid k^2-\alpha \ell^2=1\}$, the set of unitary elements in $\F_{q^2}$. If $z \in \gfw$, then we have $|S_z/\mathcal U| = 1$. If $z\not\in \gfw$ then we have $|S_z/\mathcal U|=2$. Equivalently, if $z \in \gfw$, then we have $|S_z| = q + 1$, and if $z \not\in \gfw$, then $|S_z| = 2q + 2$. 
\end{thm}

Observe that Theorem \ref{thm:densityF} states that points $z$ on the boundary generating curve have minimal density: $|S_z/\mathcal U|=1$.  This sheds light on new ways to intuitively think of finite field numerical ranges using geometry.

In Section \ref{s:prelim} we will provide some background material on finite fields and numerical ranges.  In Section \ref{s:3} we provide several foundational tools for the study of boundary generating curves over finite fields, and use this to prove Proposition \ref{thm:boundary}.  In Section \ref{s:densitygfw} we prove the special case of Theorem \ref{thm:densityF} for points on the boundary generating curve.  Next, we provide a breakdown of matrices into $q^2+2$ equivalence classes in the very short Section \ref{ss:equivalenceclasses}.  In Sections \ref{s:circles} and \ref{s:conics} we complete the proofs of Theorems \ref{thm:scalings} and \ref{thm:densityF}.  Lastly, in Section \ref{s:exceptional} we study the numerical ranges of matrices from equivalence classes that do not satisfy the conditions of our aforementioned results.

\section{Preliminaries}\label{s:prelim}

In this section we will discuss several preliminary results about conics and numerical ranges over finite fields.

\subsection{Background on Conics in Finite Fields}

The following two remarks are generalizations of \cite[Lemma 2.1]{CJKLR}. Remark \ref{lem:nonempty} from \cite{Lidl} shows existence, and Remark \ref{lem:CPoints} shows the precise number of solutions.

\begin{rem}\label{lem:nonempty} Let $c \in \F_q$ be nonzero. Then for all $k \in \F_q$, there exists $a, b \in \F_q$ such that $a^2 + c b^2 = k$.
\end {rem}

 Let our boundary generating curve be (after a linear change of variables) of the form $x^2 + cy^2 = k$. If $c$ is a nonzero square in $\F_q$, then our boundary generating curve is an ellipse. If $c$ is a nonsquare in $\F_q$, then our boundary generating curve is a hyperbola. This bring us to our computation of the number of points on each conic.

\begin {rem}\label{lem:CPoints}
All smooth conics in the projective space of $\F_q$ are projectively equivalent (\cite[Theorem 5.16(i)]{Hirschfeld}, \cite[Theorem 6.30]{Lidl}). Each conic has $q + 1$ points (\cite[Lemma 5.25(iv)]{Hirschfeld}). The line at infinity meets any given conic at 0, 1, or 2 points, where it is an ellipse if it meets the conic at 0 points, parabola at 1 point, and hyperbola at 2 points. In this paper, we will only be considering the ellipse and hyperbola cases.
\end {rem}

\subsection{Background on Finite Field Numerical Ranges}\label{ss:backgroundffnr}

We begin with two results from \cite{CJKLR} that generalize easily to $\mathbb F_{q^2}$.  The proof of the first does not depend on the field in question, so we simply state it here.  The second relies slightly on Remark \ref{lem:nonempty}.

\begin{lem}[Unitary Equivalence, generalization of Lemma 2.6 from \cite{CJKLR}]
Let $A, U \in M_n(\F_{q^2})$ with $U$ unitary. Then $W(A) = W(U^*AU)$.
\end{lem}

\begin{lem}[Schur's Theorem, generalization of Proposition 3.4 from \cite{CJKLR}]\label{lem:schur}
Let $A \in M_2(\mathbb F_{q^2})$.  {Assume further that $A$ has an eigenvalue in $\F_{q^2}$ with some eigenvector $\mf v$ satisfying $\scal{\mf v,\mf v}\neq 0$.} Then there exists a unitary matrix $U$ for which $U^*AU$ is upper triangular.
\end{lem}

\begin {proof}

Let $A \in M_2(\F_{q^2})$ with an eigenvalue $\lambda \in \F_{q^2}$. Then $(A - \lambda I)\mathbf v = 0$ has at least one family of solutions $\mathbf v$ with $\scal{\mathbf{v}, \mathbf{v}} \neq 0$, and after row reduction, we can find a solution $\mathbf v$ linearly in terms of $A$ and $\lambda$. Hence, $\mathbf v$ has entries in $\F_{q^2}$, say $\mathbf v = [a + b\beta, c + d\beta]^T$.  By Remark \ref{lem:nonempty}, there exist $e,f\in \F_q$ such that $e^2+\alpha f^2=\scal{\vv,\vv}^{-1}$.  Scaling by $e+\beta f$, we assume without loss of generality that $\scal{\mathbf{v}, \mathbf{v}} = 1$. Let $\mathbf w = [-c + d\beta, a - b\beta]^T$. Thus $\mathbf v^*\mathbf w = 0$ and $\scal{\mathbf{w}, \mathbf{w}} = 1$. Let $U$ denote the matrix that has first column $\mathbf v$ and second column $\mathbf w$. Then $U$ is unitary and $AU$ has first column $A\mathbf v = \lambda \mathbf v$. Hence $U^*AU$ has first column $\lambda U^*\mathbf v$, which is $\lambda [1,0]^T$ since $\mathbf v^*\mathbf v = 1$ and $\mathbf v^*\mathbf w = 0$. Hence, we can conclude $U^*AU$ is upper triangular.
\end {proof}

In addition, we will show that the boundary generating curve behaves nicely under unitary equivalence.  Note that the same proof holds over the complex numbers, without revision.

\begin{lem}\label{lem:schur2}
Let $A,U\in M_n(\mathbb F_{q^2})$ with $U$ unitary. Then ${F_{U^*AU}}= {F_A}$ and $\Gamma_{F_{U^*AU}}^\wedge= \Gamma_{F_A}^\wedge$.
\end{lem}
\begin{proof}
We compute
\begin{eqnarray*}
    F_{U^*AU}(x,y,t) & = & \det{\left(x\Big(\frac{(U^*AU)+(U^*AU)^*}{2}\Big)+y\Big(\frac{(U^*AU)-(U^*AU)^*}{2\beta}\Big)+tI\right)} \\
    & =    & \det{\left(U^*\left(x\Big(\frac{A+A^*}{2}\Big)+y\Big(\frac{A-A^*}{2\beta}\Big)+tI\right)U\right)} \\
    & = &  F_{A}(x,y,t).
    \end{eqnarray*}
We conclude that $\Gamma_{F_A}^\wedge = \Gamma_{F_{U^*AU}}^\wedge$.
\end{proof}

\section{Boundary Generating Curves and a Proof of Proposition \ref{thm:boundary}}\label{s:3}

In this section we will prove Proposition \ref{thm:boundary}, after developing the necessary tools.  We begin by building on a result of \cite{CDRSSY} to establish a correspondence between eigenvalues of the Hermitian part of a matrix and the multiplicity of a singularity.

\begin{lem}\label{lem:eigenvalue}
Let $A \in \MM$ with Hermitian part $H_1$.  Let $\epsilon$ be an eigenvalue of $H_1$ with eigenvector $\mathbf v$.  The base polynomial $F_A$ is nonsingular at $(1:0:-\epsilon)$ if and only if $\epsilon$ is an eigenvalue of $H_1$ of algebraic multiplicity 1.  Furthermore, these conditions imply that $\scal{\vv,\vv}\neq 0$.
\end{lem}

\begin {proof}

Let $H_1$ have eigenvalue $\epsilon$ and write $h_{ij}$ for the $ij$-entry of $H_1$. Since $H_1$ is Hermitian and $\epsilon$ is an eigenvalue, $(h_{11} - \epsilon)(h_{22} - \epsilon) = h_{12}h_{21} = |h_{12}|^2$ and $H$ has eigenvector $\textbf{v} = [-h_{12},h_{11} - \epsilon]^T$
.

Suppose that $\scal{\vv,\vv}=0$. Then $|h_{12}|^2 = (h_{11} - \epsilon)^2$, so $(h_{11} - \epsilon)^2 + (h_{11} - \epsilon)(h_{22} - \epsilon) = 0.$ We get 2 cases, namely that $\epsilon=h_{11}$ or $\epsilon = (h_{11} + h_{22})/{2}$.

In the former case, we see that this implies that $h_{12}= 0$, so the off-diagonal elements must be 0.  Thus $H_1$ has standard basis vectors $\mathbf e_1$ and $\mathbf e_2$ as its eigenvectors, and the only way that $\mathbf v$ can be an eigenvector is if $H_1=I$.  Thus $\epsilon=1$ is a double eigenvalue.

In the second case, we note that $\displaystyle\epsilon = \frac{h_{11} + h_{22}}{2} = \frac{Tr(H_1)}{2} = \frac{\epsilon + \lambda}{2}$ where $\lambda$ is a second eigenvalue of $H_1$.  Solving for $\lambda$ yields a double eigenvalue once again:  $\lambda=\epsilon$.

By \cite[Lemma 5]{CDRSSY} (all but the last sentence of which applies equally to the finite field situation), the double eigenvalue arising in either of these cases means $F_A$ has order 2 at $(1 : 0 : -\epsilon)$.

We conclude that $\scal{\vv,\vv}=0$ implies both that $F_A$ is singular at $(1:0:-\epsilon)$ and that $\epsilon$ is a double eigenvalue of $H_1$.  The contrapositive of this conclusion is exactly the last statement that we were trying to prove.

Lastly, it should be clear that \cite[Lemma 5]{CDRSSY} establishes that $\epsilon$ is a simple (multiplicity one) eigenvalue if and only if $F_A$ is nonsingular at $(1:0:-\epsilon)$.
\end {proof}

\begin{lem}\label{lem:vert}
    Let $A\in \MM$.  Let $z=r+\beta s\in \Gamma_{F_A}^\wedge$ such that the tangent line to $\Gamma_{F_A}^\wedge$ at $z$ is of the form $x=r$.  {If $F_A$ nonsingular at $(1:0:-r)$}, then $z\in W(A)$.
\end{lem}

\begin{proof}
 Let $A\in \M$ and let $z=r+\beta s\in \Gamma_{F_A}^\wedge$ such that the tangent line to $\Gamma_{F_A}^\wedge$ at $z$ is of the form $x=r$.

Since $x=r$ is a tangent line to $\gfw$, $(1:0:-r)\in \gf$.  This means that $\textmd{det}(H_1 -rI) = 0$ and therefore $r$ is an eigenvalue of $H_1$. By Lemma \ref{lem:eigenvalue}, since $F_A$ is nonsingular at $(1:0:-r)$, $H_1$ has two distinct eigenvalues and (without loss of generality) $\scal{\vv,\vv}=1$.
Therefore, $\scal{A\vv,\vv}$ is a point in $W(A)$ on the line $x=r$.  So $\scal{A\vv,\vv}=r +\beta s'$ for some $s'\in \F_q$.

It remains to show that $s' = s$, which we will prove by showing that this point $r+\beta s'$ is in $\gfw$.  That is, we will show that the line $r x + s' y + t=0$ is tangent to $\Gamma_F$ at the point $(1:0:-r)$.  In particular, we will show that the partial derivatives of $F$ with respect to $x$, $y$, and $t$ have a ratio of $(r:s':1)$ at the point $(1:0:-r)$ on $\Gamma_F$.  We begin with the partial derivative with respect to the first variable, $x$, using Jacobi's formula in the first step.

   $$F_x(x,y,t)= \tr\left(\textmd{adj}(xH_1+yH_2+t I) \frac{\partial }{\partial x}(xH_1+yH_2+t I)\right).$$

     At the point $(1:0:-r)$, this simplifies to $\tr\big(\textmd{adj}(H_1-r I) H_1\big).$

   Recall that $\scal{\vv,\vv}\neq 0$.  Using Lemma \ref{lem:schur} we assume without loss of generality that $H_1$ is diagonal and that $r$ is the first entry in $H_1$.  Recalling that $H_1$ has distinct eigenvalues, we let $\lambda$ denote the second eigenvalue so that $H_1=\textmd{diag}(r,\lambda)$.  Thus the above trace formula simplifies to $r(\lambda-r)$.

    A similar calculation yields that $F_t(1,0,-r)=\lambda-r$.  To compute $F_y(1,0,-r)$, we will note that $r+\beta s'=\scal{A\mathbf e_1,\mathbf e_1}=\mf e_1^* H_1 \mf e_1  +\beta \mf e_1^* H_2 \mf e_1$.  Since $\mf e_1$ is an eigenvector of $H_1$ corresponding to the eigenvalue $r$, we conclude that $s'=\mf e_1^* H_2 \mf e_1$ is the first coordinate of the matrix $H_2$.  Our situation is thus similar to that described above, with $F_y(1,0,-r)=\tr\big(\textmd{adj}(H_1-r I) H_2\big) = (\lambda-r)s'.$

    We conclude that $(F_x:F_y:F_t)(1:0:-r)=(r:s':1)$, so that $r+\beta s'$ is indeed on $\gfw$.

Therefore, both $r+\beta s$ and $r+\beta s'$ are on $\gfw$, both points on the line $x=r$ tangent to $\gfw$.  Then both $rx+sy + t = 0$ and $rx + s'y+t= 0$ are tangent lines to $\gf$ at the point $(1:0:-r)$.  Recall that $F_A$ is nonsingular at $(1:0:-r)$.  Thus these two tangent lines are the same and $s=s'$.
\end{proof}

Next, we recall a linearity result for the numerical range.

\begin{lem}[Lemma 2.7 from \cite{CJKLR}, Remarks 6-7 from \cite{Ballico1}]\label{lem:linearity}
Let $A\in M_n(\F_{q^2})$ and let $\rho,\tau\in \F_{q^2}$.  Then $W(\rho A+\tau I)=\rho W(A)+\tau$.
\end{lem}

We will now prove in the next two lemmas that the boundary generating curve satisfies the same linearity condition as the numerical range.  Note that our proof technique also applies in the classical situation of complex number numerical ranges.  For completeness, we state those linearity conditions here.

\begin{lem}\label{lem:hAalt}
    Let $\rho=a+\beta b\in \Fext$ be nonzero.  Then $\Gamma_{F_{\rho A}}^\wedge = \rho\Gamma_{F_{A}}^\wedge$.  Additionally, $F_A$ is nonsingular if and only if $F_{\rho A}$ is nonsingular.
\end{lem}

\begin{proof}
    We will begin by establishing some notation.   We replace each $z=x+\beta y \in \Fext$ with $\ve(z)=[x,y,1]^T$ so that we can exclusively use vector notation.     Write $B_{\rho}=
   \left[ \begin{matrix}
a & b\alpha & 0\\
 b & a & 0\\
0 & 0 & 1\\
\end{matrix}\right]
    $, so that $B_{\rho}\ve(z)=\ve(\rho z)$.
    Lastly, we write $F_{A,j}$ to denote the partial derivative of $F_A$ with respect to the $j$th variable.
    
     We will analyze points in $\Gamma_{F_{A}}^\wedge$ and $\Gamma_{F_{\rho A}}^\wedge$ by considering partial derivatives of $F_A$ and $F_{\rho A}$.  We compute that 
        \begin{eqnarray*}
            F_{\rho A}(\ve(z)) & = & \textmd{det}(x(aH_{1}+\alpha bH_{2})+y(aH_{2}+bH_{1})+tI) \\
            & = & \textmd{det}((ax+by)H_{1}+(\alpha bx + ay)H_{2}+tI)\\
            & = & F_{A}(B_{\rho}^T\ve(z)).
        \end{eqnarray*}
     We pause here to note that $F_{\rho A}$ is clearly nonsingular if and only if $F_A$ is nonsingular. 
     
     Next, to compute partial derivatives of $F_{\rho A}(\ve(z))$, we apply the multivariate chain rule to $F_{A}(B_{\rho}^T\ve(z))$.  This results in 
        \begin{eqnarray*}
            \nabla F_{\rho A}(\ve(z)) & = & B_{\rho}(\nabla F_A)(B_{\rho}\ve(z))   \\
             & = & B_{\rho}[F_{A,1}(B_{\rho}^T\ve(z)),F_{A,2}(B_{\rho}^T\ve(z)),F_{A,3}(B_{\rho}^T\ve(z))]^T.
        \end{eqnarray*}
     Therefore we have 
     $$\Gamma_{F_{\rho A}}^\wedge=B_{\rho}\{[F_{A,1}(B_{\rho}^T\ve(z)),F_{A,2}(B_{\rho}^T\ve(z)),F_{A,3}(B_{\rho}^T\ve(z))]^T: F_A(B_{\rho}^T\ve(z))=0\}.$$  Since $B_{\rho}^T$ is invertible, this is exactly $\Gamma_{F_{\rho A}}^\wedge= B_{\rho} \Gamma_{F_A}^\wedge$.  In terms of multiplication in $\Fext$, this is $\Gamma_{F_{\rho A}}^\wedge=\rho\Gamma_{F_A}^\wedge$.
\end{proof}

\begin{lem}\label{lem:translate}Let $\tau=s_1+\beta s_2 \in \F_{q^2}$.  Then $(a:b:c)\in \Gamma_{F_{A+\tau I}}^\wedge$ if and only if $(a-s_1:b-s_2:c)\in \Gamma_{F_{A}}^\wedge$.
\end{lem}

\begin{proof}We apply the same differentiation techniques as in Lemma \ref{lem:vert}.  First, assume $\tau = s_1\in \F_q$.  Then
$$(F_{A+s_1I})_x(x,y,t)=\tr\big(\textmd{adj}(x(H_1+s_1I)+yH_2+tI) \frac{\partial }{\partial x}(x(H_1+s_1 I)+yH_2+tI)\big).$$  This simplifies to $$\tr\big(\textmd{adj}(xH_1+yH_2+(\tau x+t) I)H_1\big)+\tr\big(\textmd{adj}(xH_1+yH_2+(\tau x+t)I)\big)\tau.$$  This is of course equal to $(F_A)_x(x,y,\tau x+t)+{\tau}(F_A)_t(x,y,\tau x+t)$.  Similarly, $(F_{A+\tau I})_y(x,y,t)=(F_A)_y(x,y,\tau x+t)$ and
$(F_{A+\tau I})_t(x,y,t)=(F_A)_t(x,y,\tau x+t)$.

Let $(a:b:c)\in \Gamma_{F_{A+\tau I}}^\wedge$.  Then for some $x,y,t$ we have $$(a:b:c)=\left((F_{A})_x(x,y,\tau x+t)+\tau(F_{A})_t(x,y,\tau x+t):(F_{A})_y(x,y,\tau x+t):(F_{A})_t(x,y,\tau x+t)\right).$$  Choosing $X=x, Y=y$, and $T=\tau x+t$ allows us to simplify this expression without affecting the derivatives.  Thus $(a:b:c)=(e+\tau:f:g)$ for some $(e:f:g)\in \Gamma_{F_{A}}^\wedge$.  This concludes the primary case, when $\tau\in \F_q$.

If $\tau\in \beta \F_q$, the proof proceeds similarly with only the second coordinate changing.  If $\tau=s_1+s_2\beta$ with $s_1,s_2\in \F_q$, we follow the same process in two steps, first translating by $s_1$ and then by $s_2$.
\end{proof}

We are now ready to prove Proposition \ref{thm:boundary}.

\begin{proof}[Proof of Proposition \ref{thm:boundary}]
In Lemma \ref{lem:vert} we proved that if $z\in \gfw$ with vertical tangent line, then $z\in W(A)$.  We are now ready to discuss other points $z$ on $\gfw$.  Let $L:ax+by=ct$ denote the tangent line (in projective coordinates) to $\gfw$ at $z$.

Without loss of generality, set $t = 1$ in our Equation of line $L$, and parameterize it as follows: 
\begin{equation}
    L = 
    \begin{cases}
    x(\gamma) = d\gamma+e, \\
    y(\gamma) = f\gamma+g
    \end{cases}
\end{equation}
    \noindent where $d, e, f, g, \gamma \in \F_q$. Let $\rho = \alpha f - d\beta$; as we are working in a finite field, we know that $-1, \alpha f, -d \in \F_q$. To show $\rho L$ is a line with a constant $x$ coordinate, we multiply $L$ by $\rho$: $$(\alpha f-d\beta)(x(\gamma)+y(\gamma)\beta) = (\alpha f-d\beta)((d\gamma+e)+(f\gamma+g)\beta),$$ which simplifies to $\alpha(ef - dg) + ((f^2\alpha - d^2)\gamma + fg\alpha - de)\beta$. As $d, e, f, g, \alpha \in \F_q$, we clearly see $\alpha(ef - dg)$ is a constant in $\F_q$; so $\rho L$ varies solely in terms of the y-coordinate; that is, $\rho L$ is a vertical line. By Lemma \ref{lem:hAalt}, we have that $\rho z \in \Gamma_{F_{\rho A}}^\wedge$.
    
    Lastly, since we assumed that $F_A$ is nonsingular, the second conclusion of Lemma \ref{lem:hAalt} tells us that $F_{\rho A}$ is also nonsingular.  Lemma \ref{lem:vert} then yields $\rho z \in W(\rho A)$, so by Lemma \ref{lem:linearity} we see that $\rho z \in \rho W(A)$.  Thus $z \in W(A)$. This concludes the proof. 
\end{proof}

We conclude this section with an example to illustrate how the classical complex number numerical range process of using vertical tangent lines to identify points on the boundary generating curve can fail over finite fields.  In the complex number situation, we can always rotate to make a tangent line vertical.  However, over finite fields, this process sometimes fails:  for example the point $1+2\sqrt{-1}$ in $\Z_7\left[\sqrt{-1}\right]$ cannot be rotated to land in the base field $\Z_7$, as the circle of radius-squared 5 centered at the origin does not intersect the $x$-axis.  Furthermore, there are many rotation-scalings over finite fields that result in matrices whose Hermitian parts with eigenvectors $\vv$ all satisfying $\scal{\vv,\vv}=0$.  We will explore the latter situation in the following example.

\begin{ex}\label{ex:nonnice}

Consider the matrix $A = \begin{bmatrix}
1 & 2 \\
0 & 0 \\
\end{bmatrix}$ over $\mathbb{Z}_7[\sqrt{-1}]$, so that $H_1 = \begin{bmatrix}
1 & 1 \\
1 & 0
\end{bmatrix}$ and $H_2 = \begin{bmatrix}
0 & 6\sqrt{-1} \\
\sqrt{-1} & 0 \\
\end{bmatrix}$. For any rotation-scaling $\tau=k+\ell\sqrt{-1}$ of $A$, we have that $\textmd{Re}\left((k+\ell\sqrt{-1})A\right) = kH_1- \ell H_2$.  Up to scalings, all rotation-scalings will result in a Hermitian part equivalent to either $H_1$ or $h_1-\ell H_2$ for $\ell=0,1,\ldots,6$.  $H_1$ has eigenvalues $4\pm 2\sqrt{-1}$ and eigenvectors $\mf v_\pm =\begin{bmatrix}
4\pm 2\sqrt{-1} \\
1
\end{bmatrix}$; observe that $\scal{\mf v_\pm,\mf v_\pm}=0$.  For $H_2$, we have eigenvalues $\pm 1$ and eigenvectors $\mf w_\pm =\begin{bmatrix}
\pm 1\sqrt{-1} \\
1
\end{bmatrix}$, for which $\scal{\mf w_\pm,\mf w_\pm}=2\neq 0$.  Furthermore, the two other rotation-scalings which do satisfy our condition of $\scal{\mf w_\pm,\mf w_\pm}\neq 0$ for any eigenvector $\mf v$ are $H_1 - H_2$ and $H_1 - 6H_2$. Any other rotation-scaling will result in eigenvectors $\mf v$ with $\scal{\mf w_\pm,\mf w_\pm}=0$.
\end{ex}

In the above example, the boundary generating curve is a hyperbola and the numerical range is a collection of scalings of that hyperbola.  We will conclude the section with a more illustrative example of this geometry, where the boundary generating curve is an ellipse and the rest of the numerical range is made up of scalings of this ellipse.

\begin{ex}
Consider the matrix $A= \begin{bmatrix}
1 & 1 \\
0 & 0 \\
\end{bmatrix}$ over $\mathbb Z_7[\sqrt{-1}]$.  Its numerical range is given in Figure \ref{fig:ell}, where the thick ellipse (in blue in the online version of this paper) is the boundary generating curve.  The points marked with diamonds (in red in the online version) are the eigenvalues of the matrix and the foci of this ellipse.  The other ellipses are scalings of this boundary generating curve.  Note that each ellipse has 8 points (by Remark \ref{lem:CPoints}), except the trivial scaling (by zero) of the ellipse to a point at its center, which is the average of the two foci.  Observe that it has points with vertical and horizontal tangent lines, which corresponds to $\textmd{Re }A$ and $\textmd{Im }A$ having eigenvectors $\mf v$ with $\scal{\mf v,\mf v}\neq 0$.
\begin{figure}[h]
\includegraphics[width=8cm]{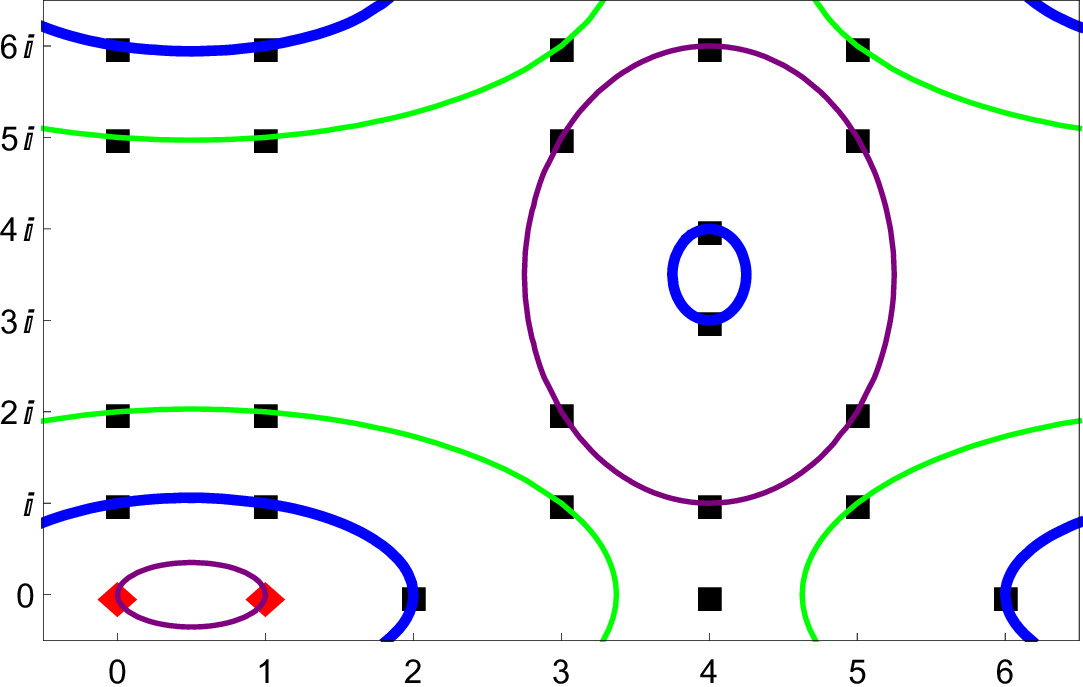}
\caption{Numerical range of $A$, with scalings of the boundary generating curve.}\label{fig:ell}
\end{figure}
\end{ex}

\section{Density on the Boundary Generating Curve}\label{s:densitygfw}

In this section we will prove the first statement in Theorem \ref{thm:densityF} about the density of points on the boundary generating curve $\gfw$.

    Let $A\in M_2(\F_{q^2})$, assume that {$F_A$ is nonsingular}.  Then by Proposition \ref{thm:boundary}, $\gfw\subseteq W(A)$. Let $z\in \gfw$ and let $S_z=\{\mathbf v\mid \scal{A \mathbf v, \mathbf v}=z, \scal{\vv,\vv}=1\}$, i.e. the pre-image of the numerical range map at $z$. 
    
    As we saw in the proof of Proposition \ref{thm:boundary}, there exists some nonzero $\rho\in \F_{q^2}$ such that $\rho z\in \Gamma_{F_{{\rho} A}}^\wedge$ is on a vertical tangent line and $F_{\rho A}$ is nonsingular.
    
    Then $\textmd{Re }\rho z$ is an eigenvalue of $\textmd{Re }\rho A$; let $\uu$ denote its eigenvector.  By Lemma \ref{lem:eigenvalue}, since $F_{\rho A}$ is nonsingular we have that $\textmd{Re }\rho z$ is a simple eigenvalue (multiplicity one) and without loss of generality $\uu$ is a unit vector.  Then $\uu^*H_1(\rho A)\uu=\Real \rho z$, so that $\Real \uu^*\rho A\uu = \Real \rho z$.  Let $L$ denote the line tangent to $\Gamma_{F_{\rho A}}^\wedge$ at $\rho z$, i.e. $L: y=\Real \rho z$.  {Then $\widehat L$ is a point on $\Gamma_{F_{\rho A}}$, with tangent line $\widehat{\rho z}$.  Since $F_{\rho A}$ is a degree 2 nonsingular curve, $L$ and $\Gamma_{F_{\rho A}}$ cannot have any other intersection points.}  Thus $\uu^*\rho A\uu = \rho z$, so $\uu\in S_{\rho z}(\rho A)$.  Observe that $\uu^* \rho A \uu = \rho z\in \Gamma_{F_{\rho A}}^\wedge$ if and only if $\uu^*Au=z\in \Gamma_{F_A}^\wedge$.  That is, we have that $\uu\in S_z(A)$.  

    By Lemma \ref{lem:schur} we assume without loss of generality that $\Real \rho A$ is a diagonal matrix, with diagonal entries $\Real \rho z$, with eigenvector $\mathbf e_1=\uu$, and some other eigenvalue $\lambda$, with unit eigenvector $\mathbf e_2$.

    Let $\ww\in S_z$.  Then $\ww=a\mathbf e_1+b\mathbf e_2$ with $|a|^2 + |b|^2 = 1$.  Then $\Real rz=\ww^*\textmd{Re}(\rho A)\ww=|a|^2\textmd{Re}(\rho z) +|b|^2 \lambda $.  Hence $\textmd{Re}(\rho z)(1-|a|^2)=\lambda |b|^2$.  Thus either $b=0$, or $\Real \rho z=\lambda$.  Since $\Real \rho A$ has distinct eigenvalues, we conclude that $b=0$.  Therefore $\ww=a\uu$ for some $a$ in the set of unitary scalars $\mathcal U$.  This completes the proof that $|S_z/\mathcal U|=1$. It then follows that $|S_z| = q + 1$.

Note that the above proof did not rely on the fact that we are working over a field of characteristic $p$, and in fact the entire proof holds for the situation of numerical ranges over the complex numbers.  We therefore have proven the following as well.

\begin {cor}\label{cor:densitygfw}

Let $A\in M_2(\C)$ and assume $F_A$ is nonsingular.  Then for each $z\in \gfw$, we have $|S_z/\mathcal U| = 1$.

\end {cor}

\section{Equivalence Classes of 2-Dimensional Matrices}\label{ss:equivalenceclasses}

In this short section, we return to some preliminaries from Section \ref{ss:backgroundffnr}, and combine it with the results of Section \ref{s:3}.  Recall that in our version of Schur's Theorem (Lemma \ref{lem:schur}), numerical ranges of matrices satisfying the conditions of Theorems  \ref{thm:scalings} and \ref{thm:densityF} are invariant under unitary transformation. By Lemma \ref{lem:schur2}, the boundary generating curve preserves this invariance.  

Similarly, from Lemma \ref{lem:linearity} we know that $W(\rho A + \tau I) = \rho W(A) + \tau$ for $\rho,\tau \in \F_{q^2}$. Therefore, after unitary transformation, multiplication, and translation, every matrix satisfying the conditions of Theorems \ref{thm:scalings} and \ref{thm:densityF} is equivalent to one of the following matrices, for nonzero $\zeta \in \mathbb F_{q^2}$.

$$\begin{bmatrix}
  0 & 0 \\
  0 & 0 \\
\end{bmatrix},
\begin{bmatrix}
  1 & 0 \\
  0 & 0 \\
\end{bmatrix},
\begin{bmatrix}
  0 & 1 \\
  0 & 0 \\
\end{bmatrix},
\begin{bmatrix}
  1 & \zeta \\
  0 & 0 \\
\end{bmatrix}.$$

  Observe that there are $q^2+2$ such equivalence classes of matrices up to unitary similarity. 
  
  Note that Lemmas \ref{lem:hAalt} and \ref{lem:translate} state that the boundary generating curve is also preserved under these transformations that reduce our matrices to these equivalence class representations.  We will study each of the situations given by these equivalence classes in our remaining sections.

\section{Circular Boundary Generating Curve}\label{s:circles}

In this section we will study matrices equivalent to $A=\begin{bmatrix}
         0 & 1 \\
         0 & 0 \\
       \end{bmatrix}$.  Let $z\in W(A)$, so there exists some unit vector $\mathbf v=\begin{bmatrix}
                                         e \\
                                         f \\
                                       \end{bmatrix}$ such that $\scal{A\mathbf v,\mathbf v}=z$.  We write $\mathbf w'$ to denote the following involution of our vector:  $\begin{bmatrix}
                                         \overline f \\
                                         \overline e \\
                                       \end{bmatrix}$.  Furthermore, let $k:=|e|^2$, so that $|f|^2 = 1-k$.  Then $z=\scal{A\mathbf v,\mathbf v}=\overline e f = \scal{A\mathbf v',\mathbf v'}$.  Observe that for fixed $k$, we have that $z$ is the set of all points on the circle centered at the origin of radius-squared $k(1-k)$. 

Firstly, when $k\neq 1-k$, we have two distinct unit vectors (which are not unitary multiples of each other) mapping to the same point $z$. Recalling that $S_z$ is the pre-image of $z$ under the numerical range map and that $\mathcal U$ is the set of all unitary scalars, we obtain that $|S_z/\mathcal U|\geq 2$ when $k\neq 1/2$.

Assume now that $k=1-k$.  Then $|e|^2 = |f|^2=1/2$, so there exists some unitary scalar $u\in \F_{q^2}$ for which $e = u \overline f$. Then we have that $\vv=u\vv'$.  Additionally, since $k=1/2$ we have that $z$ is on the circle centered at the origin of radius-squared $1/4$.  

We compute $\gfw$ in this context.  We have $$F_A(x,y,t)=\det\left(
x\begin{bmatrix}
0 & \frac{1}{2} \\
\frac{1}{2}  & 0 \\
\end{bmatrix}+
y\begin{bmatrix}
   0 & \frac{1}{2\beta} \\
   \frac{-1}{2\beta} & 0 \\
 \end{bmatrix}+
 t\begin{bmatrix}
    1 & 0 \\
    0 & 1 \\
  \end{bmatrix}
  \right)\\
  =-\frac{1}{4}\left(x^2-\frac{y^2}{\alpha}\right)+t^2.$$
To compute the dual, we begin with the partial derivatives:  $F_x(x,y,t)=-x/2$, $F_y(x,y,t)=y/(2\alpha)$, $F_t(x,y,t)=2t$.  Using our formula for $F$, we obtain $\widehat{F}(x,y,t)=\frac{t^2}{4}-x^2+\alpha y^2$.  Hence our affine curve is the circle $\Gamma_{F}^\wedge: x^2-\alpha y^2 = \frac{1}{4}$.  Now, observe that this circle contains exactly the points $z$ for which $|z|^2 = \frac{1}{4}$.  This is exactly the points described in the preceding paragraph.  By the first part of Theorem \ref{thm:densityF}, $|S_z/\mathcal U|=1$ for any $z$ on this circle.  

Recall that in our earlier case for when $k\neq 1-k$ we had $|S_z/\mathcal U|\geq 2$ and $|z|^2=k(1-k)$.  We see that these points are on scalings of the circle $\gfw$.  By \cite[Lemma 3.5]{CJKLR} and \cite[Proposition 1]{Ballico1}, the numerical range is exactly the union of $(q-1)/2$ of these circles (plus the origin, which can be considered a degenerate circle).  

Let $S$ denote the entire unit sphere in $\mathbb F_{q^2}^2$, so that $S=\bigcup_{z\in W(A)} S_z$.  From \cite[Lemma 2.4]{CJKLR} and \cite[Proof of Lemma 3]{Ballico2} we have that $|S| = q^{3}-q$ and $|\mathcal U| = q + 1$, so that $|S/\mathcal U|=q(q-1)$.  By Remark \ref{lem:CPoints}, there are exactly $q+1$ points on $\gfw$, so we have exactly $q+1$ equivalence classes (up to unitary multiples) of vectors on $S$ mapping to $\gfw$.

Now consider the $(q-3)/2$ other circles making up the numerical range, plus the origin.  As each circle has $q+1$ points, we see that there are at least $\frac{2(q+1)(q-3)}{2}+2$ equivalence classes (up to unitary scalar multiples) of vectors mapping to these points in $W(A)\setminus \gfw$. Combining these with the aforementioned $q+1$ equivalence classes of vectors mapping to $\gfw$, we have at least $q(q-1)$ such equivalence classes, and will have more if $|S_z/\mathcal U|>2$ for some $z\in W(A)$.  However, since $|S/\mathcal U|=q(q-1)$ is our upper bound, we see that $|S_z/\mathcal U|\leq 2$ for each $z\in W(A)$.

This concludes the proofs of this case within Theorem \ref{thm:scalings} and Theorem \ref{thm:densityF}.

\section{General Conical Boundary Generating Curves}\label{s:conics}
In this section we will classify the numerical range for matrices of the form $A=\begin{bmatrix}
         1 & \zeta \\
         0 & 0 \\
       \end{bmatrix}$ that have nonsingular boundary generating curves.  We will first study the boundary generating curve $\gfw$, and then discuss the density of points in the numerical range $W(A).$
       
\subsection{Scalings of the Boundary Generating Curve}\label{ss:ellipses}
       Following the same technique for computing the boundary generating curve as in Section \ref{s:circles}, we obtain the following.

\begin {eqnarray}
\label{eqn:fa2}
F_A(x,y,t)=  xt + t^2 -\frac{1}{4}x^2|\zeta|^2 + \frac{1}{4\alpha}y^2|\zeta|^2\\
\label{eqn:gfw}
\Gamma_{F}^\wedge: \frac{\left(x - \frac{1}{2}\right)^2}{1 + |\zeta|^2} - \alpha \frac{y^2}{|\zeta|^2} = \frac{1}{4}
\end {eqnarray}

Note that when $|\zeta|^2=-1$ or $|\zeta|^2 = 0$, the boundary generating curve is singular.  These situations will be discussed in the Section \ref{s:exceptional}.  In the situation that $\left(1+|\zeta|^2\right)|\zeta|^2$ is a perfect square (respectively nonsquare) in $\F_q$, observe that $\gfw$ is an ellipse (resp. hyperbola).  We will return to this in Subsection \ref{ss:7.2}.

In this section we will study scalings of the boundary generating curve in Equation \ref{eqn:gfw}  and how they relate to the numerical range. We will begin by analyzing vectors in the pre-image of the numerical range map, and which scalings of the boundary generating curve they correspond to.

Let $[a, b]^T$ be a unit vector in $\Fext^2$ that maps to some $x + \beta y$ in the numerical range $W(A)$. Then $x + \beta y = |a|^2 + \overline{a}b\zeta$. Let $k := |a|^2$ so that $|b|^2=1-k$. We can easily compute that $x = k + \textmd{Re}(\overline{a}b\zeta)$ and $y = \textmd{Im}(\overline{a}b\zeta)$. We will show that this point $(x,y)$ is on the circle centered at $(k,0)$.  We compute

\begin{eqnarray*}
(x - k)^2 - \alpha y^2 & = & \textmd{Re}(\overline{a}b\zeta)^2 - \alpha \textmd{Im}(\overline{a}b\zeta)^2 \\
& = & |\overline{a}b\zeta|^2.
\end{eqnarray*}

In conclusion, 
\begin{equation}\label{eq:solvek} 
    (x - k)^2 - \alpha y^2 - k(1 - k)|\zeta|^2 = 0.
\end{equation}

Now, let $j$ be the unique number satisfying the following equation, whose origins will become more clear in the next subsection.  
\begin{equation}\label{eq:ksimp}
    \left(\frac{|\zeta|^2 + 1}{2}\right)(k + j) - \frac{|\zeta|^2}{2}= x  
\end{equation}

We define 
\begin{equation}\label{eq:Cm}
    \mathcal{C}_m: \frac{\left(x-\frac{1}{2}\right)^2}{1+|\zeta|^2} - \alpha \frac{y^2}{|\zeta|^2} = \frac{m}{4}, 
\end{equation} so that $\mathcal C_1$ is the  conic boundary generating curve. We will use Equations \ref{eq:solvek} and \ref{eq:ksimp} to determine $m$.  

Combining Equations \ref{eq:solvek} and \ref{eq:Cm} yields the equation
$$\frac{(x-k)^2}{|\zeta|^2} - k(1-k) - \frac{(x-\frac{1}{2})^2}{1+|\zeta|^2}= -\frac{m}{4}.$$ Since we know the values of quantities $x-\frac{1}{2}$ and $x-k$ from Equation \ref{eq:ksimp}, we may substitute as follows $$\frac{\Big({j(|\zeta|^2+1)+k(|\zeta|^2-1)-|\zeta|^2}\Big)^2}{|\zeta|^2} - 4k(1-k) - \frac{\Big({(|\zeta|^2+1)(k+j)-|\zeta|^2-1}\Big)^2}{1+|\zeta|^2}= -m.$$

This simplifies considerably to the following formula for $m$. 
\begin{equation}\label{eq:m}
    m = 1-(j-k)^2\frac{|\zeta|^2+1}{|\zeta|^2}
\end{equation}

Observe that $(j-k)^2$ ranges over all squares in $\mathbb F_{q}$, including zero.  Since $m$ is linear in terms of $(j-k)^2$, we can see that there are $(q+1)/2$ possible values of $m$, and thus $(q+1)/2$ scalings of the boundary generating curve. We have proven that our general point $x+\beta y$ in the numerical range $W(A)$ is on one of the $(q+1)/2$ scalings of the boundary generating curve.  We conclude that the numerical range is a subset of these scalings.
 
  \subsection{Curve Density}\label{ss:7.2}
 
 Recall from the beginning of the previous subsection that if $|\zeta|^2 (1+|\zeta|^2)$ is a square in $\mathbb F_{q}$ then the boundary generating curve is an ellipse.  This ellipse, and all its scalings, will contain $q+1$ points by Remark \ref{lem:CPoints}.  Furthermore, in this situation we see that $m$ in Equation \ref{eq:m} can be 0, thus one of the scalings of the boundary generating curve is actually just one point (the center of the ellipse, $\frac{1}{2}+0\beta$).  If $|\zeta|^2 (1+|\zeta|^2)$ is a nonsquare in $\mathbb F_{q}$, then the boundary generating curve is a hyperbola containing $q-1$ points by Remark \ref{lem:CPoints}.  Note that Equation \ref{eq:m} is always nonzero in this situation. 
 
 We will now show that $|S_z/\mathcal U| \leq 2$ for all points $z$ in the numerical range.  Write $z=x+\beta y$ and assume for sake of contradiction there exist three distinct unit vectors $[a, b]^T, [c, d]^T$, $[e, f]^T$, that map to $z$, with none of the three vectors being unitary multiples of another. Let $k := |a|^2, j := |c|^2$, and $h := |e|^2$. We follow the same as in the previous subsection to prove Equation \ref{eq:solvek} in this situation, and by symmetry we also have $(x - j)^2 - \alpha y^2 = j(1 - j)^2|\zeta|^2$.  Combining these two equations yields 
\begin {equation*} (x - k)^2 - (x - j)^2 = \Big(k(1 - k) - j(1 - j)\Big)|\zeta|^2. \end {equation*}   This equation simplifies considerably, to result in exactly Equation \ref{eq:ksimp}.  By symmetry we also obtain a similar equation involving $k$ and $h$:
 \begin{equation*}
    \left(\frac{|\zeta|^2 + 1}{2}\right)(k + h) - \frac{|\zeta|^2}{2}= x.  
\end{equation*}
 Comparing this equation and Equation \ref{eq:ksimp} results in $h=j$.  Thus $|c|^2=|e|^2$, so there exists some unitary $u\in \F_q$ (i.e. $|u|^2=1$) for which $cu=e$.  Since $[c,d]^T$ and $[e,f]^T$ both map to $x+\beta y$ under the numerical range map, we have that $\overline c d\zeta=\overline e f \zeta$.  Hence $\overline c d =\overline c \overline u f$, so that $du=f$. We conclude that $[e,f]^T$ is a unitary multiple of $[c,d]^T$.  This contradicts the assumption that we had made for the sake of contradiction.  We conclude that $|S_z/\mathcal U| \leq 2$.
 
  We pause briefly to recall from \cite[Lemma 2.4]{CJKLR}, \cite[Proof of Lemma 3]{Ballico2} the number of vectors on the unit sphere $S$ and the number of unitary scalars $\mathcal U$.  They give us $|S| = q^{3}-q$ and $|\mathcal{U}| = q + 1$, so that the number of equivalence classes up to unitary multiples is $|S/\mathcal U|=q(q-1)$. 
  
 Recall that Theorem \ref{thm:densityF} told us that $|S_z/\mathcal U| = 1$ for any point $z$ on $\gfw$. Furthermore, the final paragraph of Subsection \ref{ss:ellipses} stated that there are $(q - 1)/2$ other scalings of this boundary generating curve inside the numerical range $W(A)$.   A priori for any $z$ on one of these scalings, we only know that $0 \leq |S_z/\mathcal U| \leq 2$. Recall that in the ellipse case, one of the ellipses will be scaled by a factor of 0, i.e. scaled down to a point.
 
 In the ellipse case, this translates into having $(q-3)/2$ nontrivial scalings of the boundary generating curve, each with $q+1$ points, that in turn each have at most two equivalence classes of vectors in $S/\mathcal U$ mapping to them.  Together with at most two equivalence classes mapping to the trivial ellipse and exactly $q+1$ equivalence classes mapping to the boundary generating curve, we have that the maximum number of equivalence classes of vectors mapping to the numerical range $W(A)$ is 
 $$|S/\mathcal U|\leq \frac{q - 3}{2}(2)(q + 1) +  2(1) + (q + 1), $$ with inequality being possible only if $|S_z/\mathcal U|<2$ for some $z\in W(A)\setminus \gfw$.  However, this number sums to $q(q-1),$ which is exactly $|S/\mathcal U|$.  
 
 In they hyperbola case, a similar computation yields the same result: 
 \begin{eqnarray*}
    |S/\mathcal U| & \leq & \frac{q - 1}{2}(2)(q - 1) + (q - 1)\\
    &= & q(q - 1)\\ & = &|S/\mathcal U|.
\end{eqnarray*}

In both cases we find that for any $z\in W(A)$ such that $z\notin \gfw$, we have $|S_z/\mathcal U| = 2$, and equivalently that $|S_z| = 2q + 2$. We conclude that $$\bigcup_{m\in S}^{} \mathcal{C}_{m} = W(A),$$ where $$S:=\left\{1-t^2\frac{(|\zeta|^2+1)}{|\zeta|^2}\ : \ t\in \F_q\right\}.$$
 
 \subsection{Proof of Theorems \ref{thm:scalings} and \ref{thm:densityF}}
 Together, Sections \ref{s:circles} and \ref{s:conics} show that the numerical range in the nonsingular case is the set of $(q + 1)/2$ multiples of the boundary generating curve (including the boundary generating curve itself), with 2 equivalence classes of vectors mapping to points off of the boundary generating curve, and 1 equivalence class mapping to each point on the boundary generating curve. Together with our proof of the first part of Theorem \ref{thm:densityF} in Section \ref{s:densitygfw}, this completes the proof of Theorem \ref{thm:scalings} and Theorem \ref{thm:densityF}.

\section{Singular Cases}\label{s:exceptional}

 We will now discuss the situations of matrices $A$ for which $F_A$ is singular.  

\subsection{Exceptional Case}
We begin by returning to the class of matrices studied in Section \ref{s:conics}, and studying what occurs when the boundary generating curve found there is singular.  The following lemma begins to describe these; Lemma \ref{lem:eigenvalue} or direct computation proves that the boundary generating curve is singular for this infinite family of matrices. 

\begin{lem}
Let $A = \begin{bmatrix}
         1 & \zeta \\
         0 & 0 \\
       \end{bmatrix}$ with $|\zeta|^2 = -1$. Then all rotation-scalings of $A$ do not have an eigenvector $\textbf{u}$ such that $\scal{\textbf{u}, \textbf{u}} \neq 0$.
\end{lem}
\begin{proof}
We have $\textmd{Re}\big((k+\beta \ell)A\big) = kH_1- \alpha\ell H_2$. The eigenvalues of this Hermitian matrix are
$\lambda_\pm = \frac{k\beta\pm \ell}{2\beta}$ and
their respective eigenvectors are
$\begin{bmatrix}
-\zeta(k\beta - \ell)  \\
k\beta + \ell
\end{bmatrix}$ and
$\begin{bmatrix}
-\zeta \\
1
\end{bmatrix}$. Both eigenvectors $\vv$ satisfy $\scal{\vv,\vv}=0$.  
\end{proof}

In \cite[Example 3.7]{CJKLR}, this class of matrices was described as having the numerical range of the entire plane $\F_{q^2}$ minus one line (see \cite[Figure 2]{CJKLR}), but the geometric description (as a collection of coincidentally non-overlapping circles) was lacking.  We provide a new version of that graphic in Figure \ref{fig:new}, where the dotted line is the boundary generating curve and the marked points are the eigenvalues of the matrix.  

\begin{figure}[h!]
  \includegraphics[width=8cm]{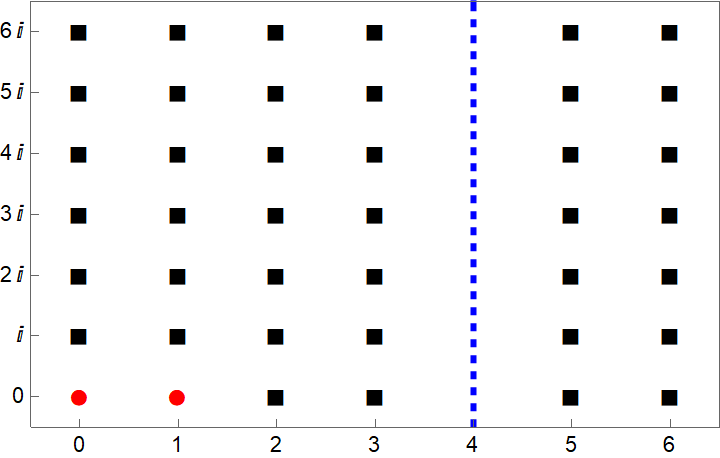}
  \caption{$\zeta=4+5\sqrt{-1}$ over $\mathbb Z_7[\sqrt{-1}]$.  The dotted line is $\gfw$, and the marked points are the eigenvalues.}\label{fig:new}
\end{figure}

We now provide a geometric interpretation of that situation, using the following proposition.  

 \begin {prop}\label{pp:zetanegone}
 Let $A = \begin{bmatrix}
         1 & \zeta \\
         0 & 0 \\
       \end{bmatrix}$. Then $|\zeta|^2 = -1$ iff $(\gfw \cup W(A)) = \F_{q^2}$ and $(\gfw \cap W(A)) = \varnothing$.
 \end {prop}
 
 \begin {proof}
  Firstly, let $\zeta$ satisfy $|\zeta|^2 = -1$. Assume for sake of contradiction that there exists some $z=x+\beta y  \in \gfw\cap W(A)$. Then there exists a vector $[a, b]^T$ with $k := |a|^2$ that maps to $z$ and satisfies Equation \ref{eq:solvek}.  Since $|\zeta|^2 = -1$, the equation simplifies to \begin{equation}\label{eq:circleexeptional}x^2 - \alpha y^2 = k(2x - 1).\end{equation}
 We compute that $\Gamma_{F}^\wedge: (x - \frac{1}{2})^2 = 0$ in this singular case, which is the line $x = 1/2$. Intersecting this with Equation \ref{eq:circleexeptional} we have $\frac{1}{4} - \alpha y^2 = 0$. Thus $\beta y=\pm 1/2$, but then $1/2\notin \F_q$.  This contradiction tells us that the numerical range and boundary generating curve share no elements.
 
 Now assume for sake of contradiction that there exist a pair of vectors $[a, b]^T$ and $[c, d]^T$, which are not unitary multiples of each other, that map to the same point $z=x+\beta y\in W(A)$.  Let $k=|a|^2$ as before, and let $j=|c|^2$.  Equation \ref{eq:ksimp} simplifies considerably in our situation to $x = 1/2$. Since this is precisely the boundary generating curve, and (from the preceding paragraph) no points are on both the boundary generating curve and in the numerical range, we arrive at a contradiction. Thus $|S_z/\mathcal{U}| \leq 1$ for any $z\notin \gfw$.  
 
 From here, we have $|\Fext| = q + (q^2 - q) = |\{(x, y) : x = \frac{1}{2}\}| + |S/\mathcal U| = |\gfw| + |W(A)|$. We conclude that $\gfw$ and $W(A)$ partition $\Fext$ into two subsets. 
 
 Lastly, we note that if $|\zeta|^2\neq -1$, then $F_A$ is nonsingular and the results from our earlier sections shows us that $\gfw\subseteq W(A)$.
 \end{proof}
 
 For matrices satisfying the $|\zeta|^2=-1$ condition of Proposition \ref{pp:zetanegone}, we can also also relate the singularities of $\gfw$ and $\Gamma_F$.  By computing partial derivatives in Equation \ref{eqn:fa2}, we find that $(-2:0:1)$ is the singular point of $\Gamma_F$.  This by itself would guarantee that the boundary generating curve $\gfw$ is a double line.

 The following corollary follows trivially.
 \begin{cor}
 Let $A = \begin{bmatrix}
         1 & \zeta \\
         0 & 0 \\
       \end{bmatrix}$. Then the following are equivalent.
       \begin{enumerate}\renewcommand{\labelenumi}{\alph{enumi})}
           \item $|\zeta|^2 = -1$ 
           \item $|S_z/\mathcal U| = 1$ when $z \notin \gfw$
           \item $S_z=\varnothing$ when $z\in \gfw$
        \end{enumerate}
 \end{cor}
 
\subsection{Unitarily Reducible Matrices}

In our final two remaining cases, we consider the matrices equivalent to $\begin{bmatrix}
         1 & 0 \\
         0 & 0 \\
       \end{bmatrix}$ and $\begin{bmatrix}
         0 & 0 \\
         0 & 0 \\
       \end{bmatrix}$.  

Let $A=\begin{bmatrix}
         1 & 0 \\
         0 & 0 \\
       \end{bmatrix}$ in $\F_{q^2}$.
      We compute
$(x + t)t = xt + t^2.$
Thus $\Gamma_F$ is the union of the affine $x+t=0$ and the line at infinity $t=0$.  Their intersection point is the singularity $(0:1:0)$, whose dual is the affine line $y=0$.  This line is exactly $\gfw$.  
The duals of the two lines $x+t=0$ and $t=0$ are respectively the points $0 + 0\beta$ and $1 + 0\beta$ on our singular boundary generating curve $\gfw$.

We will now calculate $|S_z/\mathcal U|$ for each of the points in $W(A)$, namely all of $\F_q$. The vector $[a,b]^T$ maps to $|a|^2$.  Thus $S_1=\mf e_1 \mathcal U$ and $S_0=\mf e_2 \mathcal U$.  Now let $z\in \F_q$ be other than zero and one and let $[a,b]^T\in S_z$.  Write $a=c+\beta d$, so that $z=c^2-\alpha d^2$.  The set of solutions for $c,d$ is an ellipse, so by Remark \ref{lem:CPoints}, there are $q + 1$ choices for $a$ which satisfy this equation.  (The choice of $b$ is inconsequential, as we are considering vectors up to unitary scalar multiples.)  We conclude that $|S_z/\mathcal U|=q+1.$

Overall, we have $q - 2$ points $z\in \F_q$ of density $|S_z/\mathcal U|=q+1$, two points $0,1$ of density 1, which of course adds up to the $q(q - 1)$ total equivalence classes in $S/\mathcal U$. 

Our final situation is the zero matrix. We have $$W\left(\begin{bmatrix}
         0 & 0 \\
         0 & 0 \\
       \end{bmatrix}\right)=\{0 + 0\beta\},$$ and the boundary generating curve of this zero matrix is singular and consists of exactly that one point. Since the numerical range is just the origin, we have $|S_0/\mathcal{U}| = q(q - 1)$.

This completes our classification of finite field numerical ranges of matrices within the equivalence classes listed in Section \ref{ss:equivalenceclasses}.

\bibliographystyle{amsplain}
\bibliography{master}

\end{document}